\documentclass[a4paper,12pt]{article}
\usepackage{amsmath}
\usepackage{amssymb}
\usepackage{setspace}
\usepackage{fullpage}
\usepackage{bbm}
\usepackage{bm}
\usepackage{amsthm}
\usepackage{pdfsync}
\newtheorem{theorem}{Theorem}[section]
\newtheorem{lemma}[theorem]{Lemma}

\newtheorem{corollary}[theorem]{Corollary}
\newtheorem{result}[theorem]{Result}
\newtheorem{conjecture}[theorem]{Conjecture}
\newtheorem{proposition}[theorem]{Proposition}
\theoremstyle{definition}
\usepackage{color}
\usepackage{paralist}

\newtheorem{remark}{Remark}

\def\PG{\mathrm{PG}} 
\def\AG{\mathrm{AG}} 
\def\F{\mathbb{F}}

\def\Tr{\mathrm{Tr}}

\def\la{\langle}
\def\ra{\rangle}
\def\ran{\rangle_{q^n}}

\def\S{\mathcal{S}}

\def\D{\mathcal{D}}
\def\A{\mathcal{A}}
\def\B{\mathcal{B}}

\title{The minimum size of a linear set}
\author{Jan De Beule \and Geertrui Van de Voorde}
\begin{document}
\date{}
\maketitle
\begin{abstract} 
In this paper, we first determine the minimum possible size of an $\F_q$-linear set of rank $k$ in $\PG(1,q^n)$. We obtain this result by relating it to the number of directions determined by a linearized polynomial whose domain is restricted to a subspace. We then use this result to find a lower bound on the number of points in an $\F_q$-linear set of rank $k$ in $\PG(2,q^n)$. In the case $k=n$, this confirms a conjecture by Sziklai in \cite{peter1}.
\end{abstract}

%{\bf Keywords:} Field reduction, Desarguesian spread

%{\bf Mathematics Subject Classfication:} 51E20, 51E23, 05B25,05E20

\section{Introduction}

Let $q=p^h$, $p$ prime, $h \geq 1$. The finite field of order $q$ will be denoted as $\F_q$. Let $f: \F_q \rightarrow \F_q$ be a function. The graph
of $f$ is the set of affine points $\{(x,f(x)) | x \in \F_q\}$. The following theorem expresses the state of the art
on the number of directions determined by this affine point set.

\begin{theorem}[\cite{Ball2003}]\label{th:graph1}
Let $f: \F_q \rightarrow \F_q$ be a function. Let $N$ be the number of directions determined by $f$. Let $s = p^e$ 
be maximal such that any line with a direction determined by $f$ that is incident with a point of the graph of 
$f$ is incident with a multiple of $s$ points of the graph of $f$. Then one of the following holds:
\begin{compactenum}[(i)]
\item $s = 1$ and $\frac{q+3}{2} \leq N \leq q+1$;
\item $\F_s$ is a subfield of $\F_q$ and $\frac{q}{s} + 1 \leq N \leq \frac{q-1}{s-1}$;
\item $s=q$ and $N=1$.
\end{compactenum}
Moreover, if $s > 2$, then the graph of $f$ is $\F_s$-linear.
\end{theorem}

Theorem~\ref{th:graph1} completed two unresolved cases from \cite[Theorem 1.1]{BBB}. 
Many generalizations of the questions studied in \cite{BBB} have been investigated, and it is impossible to summarize them 
in a concise way in this introduction. One notable generalization is found in \cite{peter}, where bounds on 
the number of directions determined by an affine set points of size smaller than $q$ are derived. This
paper turns out to be very useful to study the following question.

Let $n > 1$ and let $V \subset \F_{q^n}$ be a set of size $q^k$, $k \geq 1$, that is also a $k$-dimensional vector space
over $\F_q$. Let $f: V \rightarrow \F_{q^n}$ be a function that is $\F_q$ linear, i.e. $f(\lambda x + \mu y) = \lambda f(x) + \mu f(y)$
for all $x,y \in V$ and for all $\lambda,\mu \in \F_q$. Then what is the minimum number of directions determined by the {\em graph of
$f$}, i.e. the set $\{(x,f(x)) | x \in V\} \subset \AG(2,q^n)$? 

This question is motivated by the question to find a lower bound on the size of an $\F_q$-linear set of rank $k$ in $\PG(1,q^n)$. The main result
can be stated as follows.

\begin{theorem}\label{th:main_intro}
An $\F_q$-linear set of rank $k\leq n$ in $\PG(1,q^n)$ which contains at least one point of weight one, contains at least $q^{k-1}+1$ points.
\end{theorem}

For linear sets of rank $n$ in $\PG(1,q^n)$, Theorem \ref{th:main_intro} was shown in \cite[Lemma 2.2]{olga2}.

In Section~\ref{sec:preliminaries}, the connection between linear sets in $\PG(1,q^n)$ and the direction problem is described. Section~\ref{sec:line}
is devoted to the proof of Theorem~\ref{th:main_intro}, and in Section~\ref{sec:appl} we will use Theorem~\ref{th:main_intro} to 
derive a lower bound on the size of linear sets in $\PG(2,q^n)$ under certain assumptions.

\section{Preliminaries}\label{sec:preliminaries}

For any additive group $V$ let $V^* := V \setminus \{0\}$.

%{\color{red}Nog een paragraafje of twee in te vullen want nu vallen we wel heel hard met de deur in huis. Ga je gang :-)}

\subsection{Linear sets}

Let $k\geq 1$ and $r\geq 2$. A point set in $\PG(r-1,q^n)$ is an $\F_q$-linear set of rank $k$ if it equals a set $L_U$ for some $\F_q$-vector subspace $U$ of $\F_q^{rn}$ of dimension $k$, where
$$L_U=\{ \la u\ra_{q^n}\mid u \in U^\ast\}.$$ In other words, $L_U$ consists of the projective points defined by the vectors of $U^\ast$. Let $P=\la v\ra_{q^n}$ be a point of $L_U$, then the {\em weight} of the point $P$ in $L_U$ is defined as $wt(P)=\dim_q(\la v\ra_{q^n}\cap U)$. Hence, whenever we talk about the weight of a point in a linear set, the underlying defining vector space $U$ should be specified.

An equivalent point of view on linear sets and their weights is obtained using {\em field reduction}.
The underlying vector space of the projective space $\PG(r-1,q^n)$ is $V(r,q^n)$; if we consider $V(r,q^n)$ as a vector space over $\F_q$, then it has dimension $rn$, so it defines a projective space $\PG(rt-1,q)$. In this way, every point $P$ of $\PG(r-1,q^n)$ corresponds to a subspace of $\PG(rn-1,q)$ of dimension $(n-1)$ and it is not hard to see that this set of $(n-1)$-spaces forms a spread of $\PG(rn-1,q)$, which is called a {\em Desarguesian spread}. If $U$ is a subset of $\PG(rt-1,q)$, and $\S$ a Desarguesian $(n-1)$-spread, then we define $\B(U):=\{R \in \S \mid U \cap R \neq \emptyset\}$. In this paper, we consider the Desarguesian spread $\S$ as fixed and we identify the elements of $\B(U)$ with their corresponding points of $\PG(r-1, q^n)$. An \emph{$\F_q$-linear set $T$ of rank $k$} in $\PG(r-1,q^{t})$ is then a set of points such that $T=\B(\mu)$, where $\mu$ is an $(k-1)$-dimensional subspace of $\PG(rt-1,q)$. If $\mu$ is the $(k-1)$-space defined by the $k$-dimensional vectorspace $U$ of $V(rn-1,q)$, then $\B(\mu)=L_U$. The {\em weight} of a point $P=\B(p)$ of the linear set $\B(\mu)$ can then equivalently be defined as $\dim(\mu\cap \B(p))+1$, i.e., one more than the projective  dimension of the intersection of the spread element corresponding to the point $P$ with the subspace $\mu$ defining the linear set $\B(\mu)$. We see that a point $Q$ belongs to the linear set $\B(\mu)$ if and only if the weight of $Q$ in $\B(\mu)$ is at least one. For more information about linear sets and field reduction, we refer to \cite{FQ11,olga}.

\begin{remark}
 Suppose that we have a linear set $\B(\mu)$ that has only points of weight at least $j$ for some $j>1$ and contains a point $Q=\B(q)$ of weight exactly $j$. We can pick a subspace $\nu$ of codimension $j-1$ in $\mu$ meeting $\B(q)\cap \mu$ in exactly a point. As all points have weight at least $j$ and $\nu$ has codimension $j-1$ in $\mu$, $\B(\mu)=\B(\nu)$ and $Q$ has weight $1$ in $\B(\nu)$.
 
 We see that every $\F_q$-linear set $L_U$ can be written as an $\F_q$-linear set $L_{U'}$ that contains at least one point of weight one. In Theorem \ref{hoofd}, we will restrict to linear sets having a point of weight one, which should by the previous argument not be seen as a heavy restriction. However, the study of linear sets $L_U$ where all points have weight $>1$ is of interest as well (see Remark \ref{allesmeer}).
\end{remark}

\begin{remark} The only $\F_q$-linear set of rank $k>n$ in $\PG(1,q^n)$ is the set of all points of $\PG(1,q^n)$. For this reason, we restrict ourselves to $\F_q$-linear sets of rank $k\leq n$.

\end{remark}
\begin{lemma} Let $L_U$ be an $\F_q$-linear set of rank $k$ in $\PG(1,q^n)$, $k\leq n$, not containing the point $\la(0,1)\ra_{q^n}$, then $L=\{\la (x,f(x))\ra_{q^n}|x\in V^\ast\}$ for some vector subspace $V$ of dimension $k$ and some $\F_q$-linear map $f: V \rightarrow \F_{q^n}$.
\end{lemma}
\begin{proof} We have that $L_U=\{ \la u\ra_{q^n}\mid u \in U^\ast\}$, where $U$ is a subspace of dimension $k$ of $\F_q^{2n}$. We consider $\F_q^{2n}$ as $\F_{q^n}^2$ and see that every element of $U$ can be written as $(\alpha_i,\beta_i)$ for some $\alpha_i,\beta_i$ in $\F_{q^n}$, $i=1,\ldots,q^{k}$. Put $\beta_i=f(\alpha_i)$.
Suppose to the contrary that $\alpha_{i_0}=\alpha_{j_0}$ for some $i_0\neq j_0$. The elements $(\alpha_{i_0},f(\alpha_{i_0}))$ and $(\alpha_{j_0},f(\alpha_{j_0})$ are distinct elements of $U$, so if $\alpha_{i_0}=\alpha_{j_0}$, then $f(\alpha_{i_0})\neq f(\alpha_{j_0})$. As $U$ is a vector subspace, it follows that $(\alpha_{i_0},f(\alpha_{i_0}))-(\alpha_{j_0},f(\alpha_{j_0}))=(0,f(\alpha_{i_0})-f(\alpha_{j_0})$ is an element of $U$. But $L_U$ is skew from the point $\la(0,1)\ra_{q^n}$, a contradiction. We conclude that $V=\{\alpha_i \mid 1\leq i\leq q^k\}$ has size $q^k$.

Since $U$ is an $\F_q$ subspace, we have that for all $1\leq i\leq q^k$, and $\lambda,\mu\in \F_q$ that $\lambda (\alpha_i,f(\alpha_i))+\mu(\alpha_j,f(\alpha_j))=(\lambda \alpha_i+\mu\alpha_j,\lambda f(\alpha_i)+\mu f(\alpha_j))$ has to be a vector of $U$. Hence, both the set $V=\{\alpha_i \mid 1\leq i\leq q^k\}$ as the map $f$ are closed under $\F_q$-linear combinations. It follows that $V=\{\alpha_i \mid 1\leq i\leq q^k\}$ is an $\F_q$-subspace of dimension $k$ and that $f$ is an $\F_q$-linear map.

\end{proof}

From now on, whenever we write $L_U=\{\la (x,f(x))\ra_{q^n}|x\in V^\ast\}$, we assume that $U$ is the subspace $\{(x,f(x)) | x \in V\}$. In this way, the weight of a point in $L_U$ is unambiguously defined.

%\begin{lemma} An $\F_q$-linear set $S=\{\la (x,f(x))\ra_{q^n}|x\in U\}$ in $\PG(1,q^n)$ is contained in a subline $\PG(1,q^i)$, $i<n$, if and only if ...?
%\end{lemma}

%Good definition of a strictly $\F_q$-linear set. Distinction between $(x,f(x))$ with $f$ strictly $\F_q$-linear. Maybe $s=q$.

\subsection{Directions determined by a point set}

The set of directions determined by an affine point set $\A=\{\la(1,x_i,y_i)\ra_{q^n} \mid 1\leq i\leq n\}$ in $\PG(2,q^n)$ is the set $\{\la(0,x_i-x_j,y_i-y_j)\ra_{q^n}\mid 1\leq i\neq j\leq n\}$. The {\em slope} of a direction $\la (0,1,y)\ran$ is $y$, while the slope of $\la (0,0,1)\ran$ is $\infty$. If $\A$ is an affine pointset, we define $\D_\A$ to be the set of slopes of the directions determined by $\A$.

\begin{lemma}\label{richtingen} The number of points of $L=\{\la (x,f(x))\ra_{q^n}|x\in V^\ast\}$,  where $V$ is a vector subspace of $\F_{q^n}$ and $f:V\rightarrow \F_{q^n}$ is an $\F_q$-linear map, is equal to the number of directions determined by the affine pointset $\A=\{\la (1,x,f(x))\ran \mid x\in  V\}$.
\end{lemma}
\begin{proof} The number of points of $\{\la (x,f(x))\ra_{q^n}|x\in V^\ast\}=\{\la (1,f(x)/x)\ra_{q^n}|x\in V^\ast\}$ is clearly equal to the size of the set $W=\{f(x)/x|x\in V^*\}$. The points $\la(1,x_1,f(x_1))\ran$ and $\la(1,x_2,f(x_2))\ran$ determine the direction $\la(0,x_1-x_2,f(x_1)-f(x_2))\ran$. Since $f$ is $\F_q$-linear and $V$ is a subspace, $\la(0,x_1-x_2,f(x_1)-f(x_2))\ran$ is the direction $\la (0,1,f(x_3)/x_3)\ran$, with $x_3=x_1-x_2$. This implies that every direction determined by $\A$ is an element of the set $\{\la (0,1,w)\ran \mid w \in W\}$. Vice versa, take a point $\la (0,1,w_0)\ran$, with $w_0\in W$, then $w_0=f(x_0)/x_0$ for some $x_0\in V^\ast$. Then $\la (1,0,0)\ran$ and $\la (1,x_0,f(x_0))\ran$ are points of $\A$ that determine the direction $\la (0,1,w_0)\ran$. This proves that the number of directions determined by $\A$ is equal to the size of $W$.
\end{proof}
\begin{remark}
Note that the direction $\la(0,0,1)\ran$ with slope $\infty$ is not determined by $\A=\{\la (1,x,f(x))\ra_{q^n}|x\in V^\ast\}$.
\end{remark}

\subsection{The R\'edei polynomial}

Let $S=\{\la (1,x_i,y_i)\ra_{q^n}\mid1\leq i\leq | S |\}$ be a set of affine points in $\PG(2,q^n)$.
Define the R\'edei polynomial of $S$ as follows: 
$$R(X,Y)=\prod_{i=1}^{|S|} (X-x_iY+y_i).$$
As usual (see e.g. \cite{BBB,peter}), we will consider the expansion of $R(X,Y)$ using elementary symmetric polynomials. 
Let $\sigma_i(Y)$ be the $i$-th elementary symmetric polynomial of the set $\{-x_iY+y_i | 1 \leq i \leq |S|\}$,
then
$$R(X,Y)=X^{|S|} + \sum_{i=1}^{|S|}\sigma_{i}(Y)X^{|S|-i}.$$
Note that $\deg \sigma_i(Y) \leq i$. %For convenience, we define $\sigma_0(Y) = 0$. %Nuttig in het algemeen, maar blijkt toch niet nodig in onze context.

Let $y$ be a slope. Then $x$ is a root of $R(X,y) = 0$ with multiplicity $m$ if and only if the line with equation $xX_0 - yX_1 + X_2 = 0$
contains exactly $m$ points of $S$.

\section{Linear sets of $\PG(1,q^n)$}\label{sec:line}

Substituting the variable $Y$ in $R(X,Y)$ by slopes will provide particular information on the shape of the R\'edei polynomial. In the language of
direction problems, the next Lemma deals with substitution of a determined slope.

\begin{lemma}\label{Rsplits}Let $P=\la (x_0,f(x_0))\ran$ be a point of weight $j$ in $L_U=\{\la (x,f(x)\ran \mid x\in V^\ast\}$, then $R(X,y_0)$ with $y_0=f(x_0)/x_0$  is of the form
$$R(X,y_0)=\prod_{i=1}^{q^{k-j}}(X-\alpha_i)^{q^j},$$ for distinct $\alpha_i\in \F_{q^n}$.
\end{lemma}
\begin{proof} Let $P=\la (x_0,f(x_0))\ran$ be a point of weight $j$ in $L_U=\{\la (x,f(x)\ran | x\in V^\ast\}$. By definition, $P$ has weight $j$ in $L_U$ if there are $q^j$ elements $\Lambda\in \F_{q^n}$ such that $(\Lambda x, \Lambda f(x))$ is contained in $U=\{(x,f(x))|x \in V\}$. This implies that

\begin{equation}\label{eq:sol_lambda}
f(\Lambda x_0)=\Lambda f(x_0)
\end{equation}

has $q^j$ solutions for $\Lambda$. 

Let $x_1 \in V$. For any $\Lambda \in \F_{q^n}$, the point 
$\la (1,x_1 + \Lambda x_0,f(x_1) + \Lambda f(x_0))\ran \in \A \iff f(x_1 + \Lambda x_0) = f(x_1) + \Lambda f(x_0)$ 
and $x_1 + \Lambda x_0 \in V$. The condition $x_1 + \Lambda x_0 \in V$ is equivalent with $\Lambda x_0 \in V$, and so the
condition $f(x_1 + \Lambda x_0) = f(x_1) + \Lambda f(x_0)$ is equivalent with $f(\Lambda x_0) = \Lambda f(x_0)$.

Hence, the number of points of $\A$ on the line through $\la (1,x_1,f(x_1)\ran$ and $\la (0,x_0,f(x_0))\ran$ equals 
precisely the number of solutions of Equation~\ref{eq:sol_lambda} (and $\Lambda = 0$ corresponds with 
the point $\la (1,x_1,f(x_1))\ran$).

By definition, $R(X,y_0)=\prod_{x\in V}(X-xy_0+f(x))$. Now $X-xy_0+f(x)=X-x_1y_0+f(x_1)$ 
if and only if the points $\la (1,x,f(x))\ran$, $\la (1,x_1,f(x_1))\ran$, and $\la (0,1,y_0)\ran$ are collinear.
Hence, the factor $(X-x_1y_0+f(x_1))$ appears exactly $q^j$ times in $R(X,y_0)$.
\end{proof}

\begin{remark} We can also deduce Lemma \ref{Rsplits} from a more geometrical point of view. Let $L_U=\B(\pi)$, where $\pi$ is a $(k-1)$-space in $\PG(2n-1,q)$, embed $\PG(2n-1,q)$ as the subspace consisting of all points of the form $\la (0,y,z)\ra_q$ in $\PG(3n-1,q)$ and consider $L_U$ as a subset of $\PG(2,q^n)$, contained in the line $X_0=0$ (at infinity). Let $\mu$ be the subspace spanned by the point $\la (1,0,0)\ra_q$ of $\PG(3n-1,q)$ and $\pi$. Then $\B(\mu)\setminus \B(\pi)$ consists of the $q^k$ points of $\{\la (1,x,f(x))\mid x \in V\ran\}$). If $P=\la (0,x_0,f(x_0))\ran$, $x_0\in V^\ast$ is a point of weight $j$ in $L_U=\B(\pi)$, this means the spread element $S$ (of the Desarguesian $(n-1)$-spead $\S$) corresponding to $P$ meets $\pi$, and hence also $\mu$, in a $(j-1)$-dimensional space. Every line through $P$ in $\PG(2,q^n)$ containing a point $\la (1,x_0,f(x_0))\ran$ of $\la (1,x,f(x))\mid x \in V\ran$ corresponds to a $(2n-1)$-dimensional subspace of $\PG(3n-1,q)$, spanned by spread elements of $\S$, meeting $\mu$ in a subspace $\nu$ of dimension $j$. As $\pi$ is a hyperplane of $\mu$, and $P=\B(\pi\cap \nu)$ this means that the line $\B(\nu)$ contains exactly $q^j$ points of $\{\la (1,x,f(x))\mid x \in V\ran\}$. Hence every line on a point of weight $j$ of $L_U$  that contains a point of $\A$, contains exactly $q^j$ points of $\A$. From the definition of the R\'edei polynomial $R(X,Y)$, this is saying exactly that every root of $R(X,y_0)$ has multiplicity exactly $q^j$, if $y_0$ is a slope corresponding with a point of weight $j$ of $L_U$, in other words, every factor of $R(X,y_0)$ has multiplicity $q^j$.
\end{remark}

We are now ready to deduce the shape of the R\'edei polynomial of the set $\A=\{\la (1,x,f(x))\ra_{q^n} \mid x \in U\}$.

\begin{lemma}\label{vormR} If $\A=\{\la (1,x,f(x))\ra_{q^n} \mid x \in V\}$, where $V$ is an $\F_q$-vector subspace of $\F_{q^n}$ of dimension $k$
and $f: V \rightarrow \F_{q^n}$ is an $\F_q$-linear map, 
then the R\'edei polynomial of $\A$ is of the following shape:
\begin{equation}\label{eq:redei_shape}
R(X,Y)=X^{q^k}+\sigma_{q^k-q^{k-1}}(Y) X^{q^{k-1}}+\sigma_{q^k-q^{k-2}}(Y) X^{q^{k-2}}+\ldots+\sigma_{q^k-1}(Y)X\,.
\end{equation}
\end{lemma}
\begin{proof}
First consider an element $y_0 \notin \D_\A$. Then the set $V_{y_0} = \{-xy_0+f(x) | x \in V\}$ is an $\F_q$-vector subspace of $\F_{q^n}$ of dimension $k$.
Hence, by \cite[Theorem 3.52]{lidl},
\[
R(X,y_0) = \prod_{\beta \in V_{y_0}}(X-\beta)=X^{q^k}+\alpha_{1} X^{q^{k-1}}+\alpha_2 X^{q^{k-2}}+\ldots+\alpha_kX\,,
\]
with $\alpha_i \in \F_{q^n}$. 
Then consider an element $y_1\in \D_\A$. By Lemma \ref{Rsplits}, we know that if $\la (1,y_1)\ran$ is a point of weight $j_1$, then $R(X,y_1)$ contains $q^{k-j_1}$ distinct factors, each of degree $q^{j_1}$. As before, the set $V_{y_1} = \{-xy_1+f(x) | x \in V\}$ is an $\F_q$-vector subspace of $\F_{q^n}$, but the number of elements in $V_{y_1}$ is $q^{k-{j_1}}$, and hence, the dimension of $V_{y_1}$ is $q^{k-{j_1}}$. We now obtain that 

\[
R(X,y_1) = \prod_{\beta' \in V_{y_1}}(X-\beta')^{q^j}=(X^{q^{k-j_1}}+\alpha'_{1} X^{q^{k-{j_1}-1}}+\alpha'_2 X^{q^{k-2}}+\ldots+\alpha'_{k-j_1-1}X)^{q^{j_1}},
\]

We conclude that for all $y \in \F_{q^n}$, $\sigma_i(y) = 0$ if $i \not \in \{q^k-q^{j}| j = 0 \ldots k-1 \}$. Since $\deg \sigma_i(Y) \leq i$, 
each of the polynomials $\sigma_i(Y), i \not \in \{q^k-q^{j}| j = 0 \ldots k-1 \}$ has more roots than its degree, and so is identically zero.
Also note that since $\la(1,0,0)\ran \in \A$, $0 \in \{-x_iY+y_i | 1 \leq i \leq |\A|\}$, hence $\sigma_{q^k}(Y)$ is identically zero. So
$R(X,Y)$ has the shape of \eqref{eq:redei_shape}.
%%Dit was het:
%{\color{red} Note that, since $\la(1,0,0)\ran \in \A$, $\sigma_{q^k}(Y)\equiv 0$.%%%dit onderscheid is om het geval k=n te redden (we hadden q^n nulpunten voor iets van graad q^n dan voor de laatste sigma)
%%
% \sigma_k(Y) is de constante term, 
%Now pick $1\leq i\leq q^k -1$. Since $\deg \sigma_i(Y) \leq i\leq q^k-1$ and $q^k-1< q^n$, this means that $\sigma_i(Y) \equiv 0$ if $i \not \in \{q^k-q^{j}| j = 0 \ldots k-1 \}$.}
\end{proof}

\begin{remark}

A set of the form $\A=\{\la (1,x,f(x))\ra_{q^n} \mid x \in V\}$, where $f$ is an $\F_q$-linear map and $V$ is an $\F_q$-vector subspace of $\F_{q^n}$, is called an {\em affine $\F_q$-linear set} in \cite{peter}.
\end{remark}

We see that if $R(X,Y)$ is the R\'edei polynomial associated with $\{\la (1,x,f(x))\ra_{q^n} \mid x \in V\}$  
then $R(X,Y)$ is an $\F_q$-linear map in the variable $X$, and for every $y \in \F_{q^n}$, the 
map $R(X,y)$ is a linearised polynomial.

The following arguments are based on \cite{peter}. We consider the polynomial $R(X,Y)$ as a 
univariate polynomial in $X$ over the ring $\F_{q^n}[Y]$. Since 
$R(X,Y)$ is monic, division with remainder of $X^{q^n}-X$ by $R(X,Y)$ can be executed using
the ordinary Euclidean division algorithm for polynomials over a field. Hence there exists
polynomials $Q(X,Y),r(X,Y) \in \F_{q^n}[Y][X]$ such that
\begin{equation}\label{eq:euclides}
X^{q^n} - X = R(X,Y)Q(X,Y) + r(X,Y)\,,
\end{equation}
with $\deg_X r(X,Y) < \deg_X R(X,Y)$. Since $R(X,Y)$ is monic of degree $q^k$, we can write
\begin{equation}\label{eq:euclides2}
Q(X,Y) = X^{q^n-q^k} + \sum_{i=1}^{q^n-q^k}\sigma_{i}^*(Y)X^{q^n-q^k-i}\,.
\end{equation}
For convenience, we define $\sigma_0^*(Y) = 0$.
\begin{lemma}\label{le:bounded_degree}
Consider the polynomials $Q(X,Y)$ and $r(X,Y)$ from Equation~\ref{eq:euclides}. Then $\deg Q(X,Y) \leq q^n$ 
and $\deg r(X,Y) \leq q^n$ (where $\deg Q(X,Y)$ means the total degree). Furthermore, in Equation~\ref{eq:euclides2}, $\deg \sigma_i^*(Y) \leq i$.
\end{lemma}
\begin{proof} 
Before starting the Euclidean division with remainder algorithm, $Q(X,Y)$ is initialized as $0$ and $r(X,Y)$ is initialized as $X^{q^n}-X$.
So let
\begin{equation}\label{eq:rhos}
r(X,Y) = \sum_{i=0}^{q^n}\rho_{i}(Y)X^{q^n-i}\,.
\end{equation}
Then initially, $\rho_0(Y) = 1$, $\rho_{q^ n-1}(Y)=-1$, and $\rho_i(Y)=0$ for $i \not \in \{0,q^n-1\}$, so initially
$\deg \rho_{i}(Y) \leq i$ and $\deg r(X,Y)=\deg_X r(X,Y) = q^n$. As induction hypothesis, we assume that after execution of
step $j-1\geq 0$ in the Euclidean algorithm, $\deg r(X,Y) \leq q^n$, $\deg_X r(X,Y) \leq q^n - j$, and $\deg \sigma_{i}^*(Y) \leq i$ for all $i \leq j-1$.

During step $j$ of the algorithm, (1) $\sigma_j^*(Y)$ is computed and (2) $r(X,Y)$ is changed.

(1) The polynomial $\sigma_{j}^*(Y)$ becomes the leading coefficient of $r(X,Y)$ if $\deg_X r(X,Y) = q^n-j$ (because $R(X,Y)$ is monic),
and $0$ otherwise. From the induction hypothesis, $\deg r(X,Y) \leq q^n$ and $\deg_X r(X,Y) \leq q^n-j$, so in both cases $\deg \sigma_j^*(Y) \leq j$.

%From the induction hypothesis, $\deg_X r(X,Y) \leq q^n-j$, and %{\color{red} the leading coefficient of} 
%nee, de leading coeff is ten opzichte van X. Dit is een polynoom in Y, hetgeen precies \sigma_j^*(Y) zal zijn.
%$\sigma_{j}^*(Y)$ becomes the leading coefficient of $r(X,Y)$ if $\deg_X r(X,Y) = q^n-j$ (because $R(X,Y)$ is monic),
%and $0$ otherwise. So in both cases $\deg \sigma_j^*(Y) \leq j$, since $\deg r(X,Y) \leq q^n$.

(2) The remainder $r(X,Y)$ becomes $r(X,Y) - \sigma_{j}^*(Y)X^{q^n-q^k-j}R(X,Y)$. 
% Volgens mij is dit ok.
Since $\deg R(X,Y) = q^k$ and by the induction hypothesis, $\deg \sigma_j^*(Y)\leq j$ and $\deg r(X,Y)\leq q^n$,
the total degree of $r(X,Y)$ remains bounded by $q^n$. Clearly, after executing step $j$, 
$\deg_X r(X,Y) \leq q^n-(j+1)$, so $\deg \rho_i(Y) \leq i$ for all admissible $i$.

By induction we can now conclude that after execution of the algorithm, $\deg r(X,Y) \leq q^n$, $\deg Q(X,Y) \leq q^n$,
$\deg \sigma_i^*(Y) \leq i$, and $\deg \rho_i(Y) \leq i$, for all $i$.
\end{proof}

As in \cite{peter}, define $H(X,Y) = -r(X,Y)-X$, then 
\begin{equation}\label{pol:h}
X^{q^n} - X = R(X,Y)Q(X,Y) -H(X,Y) - X\,.
\end{equation}

\begin{corollary}\label{graadhs}
Consider the polynomial $H(X,Y)$ from Equation~\ref{pol:h}. Then $\deg_X H(X,Y) \leq q^k-1$. Let 
\[
H(X,Y) = \sum_{i=0}^{q^n}h_{i}(Y)X^{q^n-i}\,,
\]
then $\deg h_i(Y) \leq i$.
\end{corollary}
\begin{proof}
This follows from $H(X,Y) = -r(X,Y)-X$ and $\deg r(X,Y) \leq q^n$ by Lemma~\ref{le:bounded_degree}.
\end{proof}

\begin{remark}
Since $\deg_X H(X,Y) \leq q^k-1$, the polynomials $h_i(Y)$ are identically zero for $i \in \{0,\ldots,q^n-q^k+1\}$. In \cite{peter}, it is also mentioned
how the coefficient polynomials $h_i(Y)$ can be computed from the polynomials $\sigma_i(Y)$ and $\sigma_i^*(Y)$.
\end{remark}

The following lemma is essentially Lemma 15 from \cite{peter}, where the authors prove a similar theorem assuming $\infty \in \D_\A$, whereas in our case $\infty\notin \D_\A$.

\begin{lemma} \label{graad} Let $R(X,Y)$ be the R\'edei polynomial of the point set $\A=\{\la(1,x,f(x))\ra_{q^n}\mid x\in V\}$,
and $H(X,Y)$ the polynomial defined in Equation~\ref{pol:h}. Then 
the number of points in $L_U=\{\la (x,f(x)\ra_{q^n}\mid x \in V^*\}$ is at least $\deg_X H(X,Y)$.
\end{lemma}
\begin{proof} 
By Lemma \ref{richtingen}, the number of points in $L_U$ is the number of 
directions determined by the point set $\A=\{\la (1,x,f(x))\ra_{q^n} \mid x \in V\}$, 
where $f$ is an $\F_q$-linear map and $V$ is an $\F_q$-vector 
subspace of $\F_{q^n}$. %Let $k$ be the dimension of $V$.

Recall that the slope $y \in \D_\A$ corresponds with the direction $\la (0,1,y) \ra_{q^n}$, and that
$\la (0,0,1) \ra_{q^n}$ is not a direction determined by $\A$. Assume now that $y \not \in \D_\A$, then
$R(X,y) \mid X^{q^n}-X$, hence $H(X,y) = -X$, from which it follows that $h_i(y) = 0$ for all $i \neq q^n-1$.

Now assume that $y \in \D_\A$. Then $R(X,y)  \nmid X^{q^n}-X$, so there exists an index $j \neq q^n-1$ 
such that $h_j(y) \neq 0$, hence $h_j(Y) \not \equiv 0$. Define $i_0$ to be the smallest index such that 
$h_{i_0}(Y) \not \equiv 0$, then \begin{equation}i_0 = q^n - \deg_X H(X,Y).\label{g}\end{equation} The polynomial $h_{i_0}(Y)$ has at least $q^n - |\D_\A|$ roots,
so \begin{equation}\deg h_{i_0}(Y) \geq q^n - |\D_\A|.\label{h}\end{equation}

From Corollary \ref{graadhs}, we have that $\deg h_{i_0}\leq i_0$, which implies that $q^n \geq \deg (X^{q^n-i_0}h_{i_0}(Y))$. By Equation \ref{h}, 
$$\deg (X^{q^n-i_0}h_{i_0}(Y))= q^n-i_0+\deg h_{i_0}(Y) \geq 2q^n - |\D_\A|-i_0.$$ 
Combining the above two inequalities, we obtain that $q^n\geq 2q^n - |\D_\A|-i_0$ and we find by Equation \ref{g} that $|\D_\A|\geq q^n - i_0 = \deg_X H(X,Y)$.

\end{proof}

\begin{remark} From Lemma \ref{graad}, we could have deduced that the number of points in $L_U$ is at least $\deg_X H(X,Y)+1$ (just like in \cite{peter}) instead of the slightly weaker lower bound $\deg_X H(X,Y)$ that we now have. However, in order to obtain this impovement, we would have needed to transform our point set so that $\infty$ is a determined direction. While it is perfectly possible to do so, we chose to avoid doing this as in the proof of Theorem \ref{hoofd}, we will obtain the same lower bound $\deg_X H(X,Y)+1$ anyhow.
\end{remark}

We will need the following result, which easily follows from the geometric point of view on $\F_q$-linear sets.
\begin{result}\label{olga}\cite{olga}
The number of points in an $\F_q$-linear set is congruent to $1$ mod $q$.
\end{result}

\begin{theorem}\label{hoofd} Let $L_U=\{\la (x,f(x))\ra_{q^n}\mid x \in V^*\}$, where $V$ has dimension $k$, be an $\F_q$-linear set in $\PG(1,q^n)$ of rank $k$ which 
contains at least one point of weight one, then the size of $L_U$ is at least $q^{k-1}+1$.
\end{theorem}
\begin{proof} 
With $R(X,Y)$ the R\'edei-polynomial of $\A=\{\la (1,x,f(x))\ran \mid x \in V^*\}$, and $H(X,Y)$ defined as in \eqref{pol:h}, 
by Lemma \ref{graad}, we know that the number of points in $L_U$ is at least $\deg_X H(X,Y)$. 
Let $P=\la (x_0,f(x_0)\ran$ be a point of weight one in $L_U$. By Lemma \ref{Rsplits}, $R(X,y_0)$ 
with $y_0=f(x_0)/x_0$ splits in factors of degree $q$, and since $R(X,y_0)$ has degree $q^k$, there are $q^{k-1}$ 
different factors, each of the form $(X-\alpha_i)^q$ for some $\alpha_i\in \F_{q^n}$, $i=1,\ldots,q^{k-1}$. 
Since $X-\alpha_i$ divides $X^{q^n}-X$, it divides $H(X,y)-X$ as well. As we have found at least 
$q^{k-1}$ different linear factors dividing $H(X,y)-X$, this implies that $\deg_X H(X,Y)$ is at least 
$q^{k-1}$. We conclude that the number of points in $L_U$ is at least $q^{k-1}$, and hence, 
by Lemma \ref{olga}, at least $q^{k-1}+1$.
\end{proof}

%In the proof of Theorem \ref{hoofd}, we use the lower bound on $deg_X H(X,Y)$ which is $q^{k-1}$, in order to find a lower bound on the number of points in an $\F_q$-linear set.

In Theorem \ref{hoofd}, we find that the number of points in an $\F_q$-linear set of rank $k$ in $\PG(1,q^n)$, containing a point of weight one, is at least $q^{k-1}+1$. In the following proposition, we see that we can always find an example of such an $\F_q$-linear set, and hence, that this lower bound is sharp.
\begin{proposition}\label{vbtrace} Let $2\leq k\leq n$. There exists an $\F_q$-linear set of rank $k$ in $\PG(1,q^n)$ with $q^{k-1}+1$ elements.
\end{proposition}
\begin{proof} As usual, consider the Desarguesian $(n-1)$-spread $\S$ in $\PG(2n-1,q)$. Take a $(k-2)$-space $\mu$ contained in a spread element $S_1$ of $\S$ and let $\pi$ be a $(k-1)$-space meeting $S_1$ exactly in $\mu$. Then $\B(\mu)$ has size $q^{k-1}+1$.
\end{proof}

\begin{remark} An example of a set $\B(\pi)$ from Proposition $\ref{vbtrace}$ can be obtained using coordinates as follows: take $\alpha_0=1,\alpha_1,\ldots,\alpha_{n-k}$ to be $\F_q$-linearly independent elements of $\F_{q^n}$, let $V$ be the vector space of $\F_{q^n}$ defined by $\Tr(\alpha_i x)=0$, for $i=1,\ldots,n-k$ and put $L_U=\{\la (x,\Tr(x))\ran \mid x\in V^*\}$.

However, not every $\F_q$-linear set of size $q^{k-1}+1$ arises as in Proposition \ref{vbtrace}. For example, in $\PG(1,q^4)$, it is possible to find two non-equivalent $\F_q$-linear sets of rank $4$, each containing $q^3+1$ points (see Example B1 and C12 of \cite{olga2}). The example of Proposition \ref{vbtrace} arises as $\B(\pi)$, where $\pi$ is a $3$-space meeting one element of the Desarguesian $3$-spread $\S$ of $\PG(7,q)$ in a plane, and $q^3$ other elements in a point. The other example arises as $\B(\pi)$ where $\pi$ meets $q+1$ elements of a regulus of $\S$ in a line and $q^3-q$ others in a point.
\end{remark}

\begin{remark} In \cite[Lemma 2.2]{olga2}, the authors prove that a linear set $L_U$ of rank $n$ in $\PG(1,q^n)$ containing at least one point of weight $1$ has size at least $q^{n-1}+1$, and they show that $U$ is spanned by the vectors of $U$ defining the points of weight one in $L_U$. 
Now consider a linear set $L_U$ of rank $k$ in $\PG(1,q^n)$ containing at least one point of weight one. By Theorem \ref{hoofd}, $L_U$ has at least $q^{k-1}+1$ points. Using this result, it is easy to see that the proof of the second part of \cite[Lemma]{olga2} goes through for $k<n$, and we obtain that also in this case, $U$ is spanned by the vectors defining the points of weight $1$.
\end{remark}

Looking at the proof of Theorem \ref{hoofd}, one might think that for $\F_q$-linear sets of rank $k$ with more than $q^{k-1}+1$ points, the lower bound on $\deg_X H(X,Y)$ could be improved. This is not the case: we will show in Corollary \ref{graadzelfde} that $\deg_X H(X,Y)$ is independent of the choice of the $\F_q$-linear set of rank $k$ (as long as it has a point of weight one).

For this, we need the {\em symbolic} product of linearised polynomials, which is defined as their composition and denoted by $\circ$. More precisely, let $F(x)$ and $G(x)$ be two $\F_q$-linearised polynomials, then 
\[ (F\circ G)(x):=F(G(x))\ \mod x^{q^n}-x.\]
Unlike the ordinary product of two linearised polynomials, the composition of two linearised polynomials is again a linearised polynomial. A linearised polynomial $G(x)$ is called a (right) symbolic divisor of a linearised polynomial $F(x)$ if $F(x)=Q(x)\circ G(X)$ for some linearised polynomial $Q(x)$. With respect to symbolic (right) division, one can execute Euclid's algorithm (see \cite{ore}). So for any two linearised polynomials $F(x)$ and $G(x)$ with $\deg(G)\leq \deg(F)$, there are linearised polynomials $Q(x)$ and $H(x)$ with $\deg(H)<\deg(G)$, such that
\begin{equation}F(x)=Q(x)\circ G(x)+H(x).\label{ore}\end{equation}

\begin{proposition}\label{machtvanq} Let $\A=\{\la (1,x,f(x))\ran \mid x\in V\}$, where $f$ is an $\F_q$-linear map and $V$ is a $k$-dimensional subspace of $\F_{q^n}$. Let $R$ be the R\'edei-polynomial of $\A$, and $Q$ and $H$ as in Equation \ref{pol:h}. Then
%Suppose that we have $X^{q^n}-X=R(X,Y)Q(X,Y)-H(X,Y)-X$, where $R(X,Y)$ is a polynomial for which $R(X,y)$ is a linearised polynomial for all choices of $y$. Then $H(X,y)$ is a linearised polynomial as well and 
$\deg_X H(X,Y)$ is a power of $q$.
\end{proposition}
\begin{proof}

Pick an element $y\in \F_{q^n}$ and write $R_y(X)=R(X,y)$. By Lemma \ref{vormR}, $R(X,y)$ is a linearised polynomial. 
Note that $X^{q^n}-X$ is a linearised polynomial as well. 
So we can symbolically divide $X^{q^n}-X$ by $R_y(X)$ and find (see Equation \ref{ore}):

$$X^{q^n}-X=\widetilde{Q}_y(X)\circ R_y(X)+\widetilde{H}_y(X)=\widetilde{Q}_y(X)\circ R_y(X)-H_y'(X)-X,$$ for some linearised polynomials $\widetilde{Q}_y(X)$ and $\widetilde{H}_y(X)$ where $\deg H'\leq \deg \widetilde{H}<\deg R_y$. Note that $H'_y(X)=-\widetilde{H}_y(X)-X$ is a linearised polynomial as well and that the polynomials $\widetilde{Q}_y$ and $H_y'$ are dependent on the choice of $y$.

Write $\widetilde{Q}_y(X)=\sum_{i=0}^{n-k} \widetilde{Q}_{y,i} X^{q^i}$.
Now \begin{align*}
X^{q^n}-X&=\widetilde{Q}_{y}(X)\circ R_y(X)-H_y'(X)-X\\
&=\widetilde{Q}_{y}(R_y(X))-H_y'(X)-X\\
&=\sum_{i=0}^{n-k} \widetilde{Q}_{y,i}(R_y(X))^{q^i}-H_y'(X)-X\\
&=R_y(X)\sum_{i=0}^{n-k} \widetilde{Q}_{y,i}(R_y(X))^{q^i-1}-H_y'(X)-X
\end{align*} 
We find that 
$$X^{q^n}-X=R_y(X)\sum_{i=0}^{n-k} \widetilde{Q}_{y,i}((R_y(X))^{q^i-1}-H'_y(X)-X=R_y(X)Q'_y(X)-H'_y(X)-X$$
where we defined $Q'_y(X)=\sum_{i=0}^{n-k} \widetilde{Q}_{y,i}((R_y(X))^{q^i-1}$.

However, we know that, for a fixed $y$, 
$$X^{q^n}-X=R_y(X)Q(X,y)-H(X,y)-X.$$

This implies that for every $y\in \F_{q^n}$, 
$$X^{q^n}-X=R_y(X)Q(X,y)-H(X,y)-X=R_y(X)Q'_y(X)-H'_y(X)-X, $$ or
$$R_y(X)(Q (X,y)-Q'_y(X))=H(X,y)-H'_y(X).$$

But, if the polynomials in this equation are non-zero polynomials, then the degree of the left hand side is at least $\deg_X R(X,Y)$ 
and the degree of the right hand side is less than $\deg_X R(X,Y)$ since $H$ and $H'$ have degrees less than $\deg_X R(X,Y)$. 
Hence, for all $y$,

$$Q(X,y)=Q'_y(X)\ \mathrm{and}\ H(X,y)=H'_y(X).$$

Since for all $y\in \F_{q^n}$, $H(X,y)=H'_y(X)$ is a linearised polynomial in the variable $X$, and 
$\deg H(X,Y)<q^n$ (see Corollary \ref{graadhs}), we have that $H(X,y)$ is a linearised polynomial 
which means that $\deg_X H(X,Y)=q^i$ for some $i$.
\end{proof}

\begin{corollary} \label{graadzelfde} Let $L_U=\{\la(x,f(x))\ran\mid x \in V^*\}$ be an $\F_q$-linear set of rank $k$ containing a point of weight $1$. Let $\A=\{\la (1,x,f(x))\ran \mid x\in V\}$, let $R$ be the R\'edei polynomial of $\A$ and let $Q$ and $H$ be as in Equation \ref{pol:h}. Then $\deg_X H(X,Y)=q^{k-1}$.
\end{corollary}
\begin{proof}
By Proposition \ref{machtvanq}, $\deg_X H(X,Y)$ is a power of $q$. But $\deg_X H(X,Y)<\deg_X R(X,Y)=q^k$, which shows that $\deg_X H(X,Y)\leq q^{k-1}$. In the proof of Theorem \ref{hoofd}, we have seen that $\deg_X H(X,Y)$ is at least $q^{k-1}$. Hence, $\deg_X H(X,Y)=q^{k-1}$.
\end{proof}

\section{The size of an $\F_q$-linear (blocking) set in $\PG(2,q^n)$}\label{sec:appl}

In this section, we will extend the results found for linear sets on a line in Theorem \ref{hoofd}, to linear sets in a plane.

\begin{theorem}\label{hoofd2} Let $L$ be an $\F_q$-linear set of rank $k$ in $\PG(2,q^n)$ such that there is at least one line of $\PG(2,q^n)$ meeting $L$ in exactly $q+1$ points, then $L$ contains at least $q^{k-1}+q^{k-2}+1$ points.
\end{theorem}
\begin{proof} %First note that $k-r+1\geq 1$ since $k\geq r$: if $L_U$ spans an $r$-dimensional space over $\F_{q^n}$, then $r$ vectors of $U$ spanning $\PG(r-1,q^n)$ are certainly $\F_q$-independent, so $k\geq r$.

%We proceed by induction on $r$. First let $r=3$.

%Now assume that the theorem holds for all $\F_q$-linear sets in $\PG(r_0-1,q^n)$, containing a $(q+1)$-secant, then we will show it holds for $\F_q$-linear sets in $\PG(r_0,q^n)$. Suppose that $L=\B(\pi)$ is a linear set of rank $k$ in $\PG(r_0,q^n)$ containing a $(q+1)$ secant. Let $p_1$ and $p_2$ be points of $\pi$ such that $\B(p_1p_2)$ meets $\B(\pi)$ $q+1$ points (these are then exactly the $q+1$ points of the form $\B(r)$ with $r$ a point of the line $p_1p_2$.
%Consider a hyperplane $\Pi$ of $\PG(r_0,q^n)$ through $p_2$, not containing $p_1$. The hyperplane $\Pi$ corresponds to a $r_0n-1$-space $\bar{\Pi}$ spanned by spread elements of a Desarguesian spread in $\PG(r_0n+n-1,q)$. Project $\pi$ from $p_1$ onto $\bar{Pi}$. Since $\B(p_1)$ is a point of weight one, the spread element through $p_1$ meets $\pi$ in a point. Hence, the projection $\mu$ of the $(k-1)$-space $\pi$ from $p_1$ onto $\bar{Pi}$ is $(k-2)$-dimensional. Let $\pi'$ be the $(k-1)$-dimensional subspace of $\PG(r_0n+n-1,q)$ spanned by $p_1$ and $\mu$. Consider the $\F_q$-linear set $\B(\pi')$, then $\B(\pi')$ contains $q^{k-1}+|\B(\mu)|$ points. By Lemma \ref{}

As usual, let $\S$ be the Desarguesian $(n-1)$-spread in $\PG(3n-1,q)$. Recall that we identify a point of $\PG(2,q^n)$ with its corresponding element of $\S$.

Let $L=\B(\pi)$, where $\pi$ is a $(k-1)$-dimensional subspace of $\PG(3n-1,q)$.
Let $p_1$ and $p_2$ be points of $\pi$ such that the line $T=\la\B(p_1),\B(p_2)\ra$ of $\PG(2,q^n)$ meets $\B(\pi)$ in $q+1$ points (these are then exactly the $q+1$ points of the form $\B(r)$ with $r$ a point of the line $p_1p_2$ in $\PG(3n-1,q)$). Let $\widetilde{T}$ be the $(2n-1)$-space of $\PG(3n-1,q)$ corresponding to $T$, i.e., the subspace of $\PG(3n-1,q)$ spanned by the spread elements $\B(p_1)$ and $\B(p_2)$.

Consider a line $M$ of $\PG(2,q^n)$ through the point $\B(p_2)$, but not containing the point $\B(p_1)$. Then $M$ corresponds to a $(2n-1)$-dimensional subspace $\widetilde{M}$ of $\PG(3n-1,q)$ which is spanned by spread elements of $\S$. Project $\pi$ from $p_1$ onto $\widetilde{M}$. Since $\B(p_1)$, the spread element through $p_1$, meets $\pi$ in a point, the projection $\mu$ of the $(k-1)$-space $\pi$ from $p_1$ onto $\widetilde{M}$ is $(k-2)$-dimensional. Let $\pi'$ be the $(k-1)$-dimensional subspace of $\PG(3n-1,q)$ spanned by $p_1$ and $\mu$. The $\F_q$-linear set $\B(\pi')$ clearly contains $q^{k-1}+|\B(\mu)|$ points. Now $\B(\mu)$ contains the point $\B(p_2)$. The intersection of the spread element through $p_2$ with $\mu$ is precisely the projection of the intersection of the $(2n-1)$-dimensional space $\widetilde{T}$ with $\pi$. We know that $\widetilde{T}$ meets $\pi=\la p_1,\mu\ra$ in the line $p_1p_2$ so it follows that the spread element through $p_2$ meets $\mu$ only in the point $p_2$. Hence, $\B(p_2)$ is a point of weight $1$ in $\B(\mu)$ and by Theorem \ref{hoofd}, $\B(\mu)$ has size at least $q^{k-2}+1$. This shows that $\B(\pi')$ has at least $q^{k-1}+q^{k-2}+1$ points.

Now consider a line $N$ of $\PG(2,q^n)$ through $\B(p_1)$ and a point of $\B(\pi)$, different from $\B(p_1)$. The number of points of $\B(\pi)$ on $N$ is the number of points of $\B(\pi\cap \widetilde{N})$, where $\widetilde{N}$ is the $(2n-1)$-dimensional subspace of $\PG(3n-1,q)$ corresponding to $N$. Let $\nu=\pi\cap \widetilde{N}$, and suppose that $\nu$ is $r$-dimensional, then $\B(p_1)$ is a point of weight $1$ in $\B(\nu)$ and by Theorem \ref{hoofd}, $\B(\nu)$ has at least $q^{r-1}+1$ points. Note that $\nu$ is $r$-dimensional, and hence, $\pi'$ meets $\widetilde{N}$ in an $r$-dimensional space. By construction, this means that $\pi'\cap \widetilde{N}\cap \widetilde{M}$ is $(r-1)$-dimensional, and $\B(\pi'\cap \widetilde{N})$ has exactly $q^{r}+1$ points.  Hence, for every line $N$ through $\B(p_1)$, the number of points of $\B(\pi)\cap N$ is at least the number of points of $\B(\pi')\cap N$. We conclude that the number of points in $\B(\pi)$ is at least the number of points in $\B(\pi')$, which is at least $q^{k-1}+q^{k-2}+1$.

\end{proof}

Just as in Proposition \ref{vbtrace}, it is easy to see that the bound in Theorem \ref{hoofd2} is sharp.
\begin{proposition}\label{vbvlak} Let $3\leq k\leq n$. There exists an $\F_q$-linear set of rank $k$ in $\PG(2,q^n)$ with $q^{k-1}+q^{k-2}+1$ elements.
\end{proposition}
\begin{proof} Let $\S$ be the Desarguesian $(n-1)$-spread in $\PG(3n-1,q)$. Let $\mu$ be a $(k-3)$-space of a spread element $\B(p)$ of $\S$. Consider a line $N$ in $\PG(2,q^n)$ skew from $\B(p)$ and let $\widetilde{N}$ be the $(2n-1)$-dimensional subspace of $\PG(3n-1,q)$ corresponding to $N$. Let $\ell$ be a line of $\widetilde{N}$. Then $\la \mu,\ell\ra$ is a $(k-1)$-dimensional subspace of $\PG(3n-1,q)$, so $\B(\la \mu,\ell\ra)$ is an $\F_q$-linear set of rank $k$. By construction, it spans $\PG(2,q^n)$ and contains $q^{k-1}+q^{k-2}+1$ points.
\end{proof}

\subsection*{Concluding remarks}
One particular instance for which the question of finding the minimum size of a linear set is very relevant, is for linear sets of rank $n$ in $\PG(2,q^n)$. In this case, the $\F_q$-linear set defines a minimal {\em blocking set}. A blocking set in $\PG(2,q^n)$ is a set $B$ of points such that every line of $\PG(2,q^n)$ meets $B$ in at least $1$ point. If a blocking set in $\PG(2,q^n)$ does not contain a line, it is called {\em non-trivial} and if it contains less than $3(q^n+1)/2$ points, it is called {\em small}. It is conjectured (see \cite[Conjecture 3.1]{peter1}) that all {\em small} minimal blocking sets in $\PG(2,q^n)$ are $\F_p$-linear sets if $q$ is a power of the prime $p$. In the same paper, the author conjectures the following:

\begin{conjecture} \cite[p.1170]{peter1}\label{conjecturemin}
Let $p$ be a prime. If $\F_{p^e}$ is the ``maximum field of linearity'' then a non-trivial blocking set in $\PG(2,p^t)$, with $t=en$, has at least $(p^e)^n+(p^e)^{n-1}+1$ points. 
\end{conjecture}

The notion ``maximum field of linearity'' is used by Sziklai to indicate the following: the maximum field of linearity of a blocking set in $\PG(2,p^t)$ is $\F_{p^e}$ if and only if every line meets the blocking set in $1$ mod $p^e$ points, but not every line meets in $1$ mod $p^{e+1}$ points. The fact that $e$ is a divisor of $t$, and hence, that there is a subfield $\F_{p^e}$ of $\F_{p^t}$ follows from his work on blocking sets, but it does not necessarily hold for linear sets in general (see Remark \ref{ambetant}).

%Note that, for general sets, from the condition ``every line meets a set in $1$ mod $p^e$ points, but not every line meets in $1$ mod $p^{e+1}$ points'' does not need to follow that $e$ is a divisor of $t$, and hence, that $\F_{p^t}$ has a subfield $\F_{p^e}$.

\begin{remark} In Theorem \ref{hoofd2}, we proved that an $\F_q$-linear set of rank $k$ that contains a $(q+1)$-secant, contains at least $q^{k-1}+q^{k-2}+1$ points. It is clear if an $\F_q$-linear set contains a $(q+1)$-secant, then the maximum field of linearity is indeed $\F_q$. In \cite[Corollary 5.2]{peter1}, the author also shows the converse for blocking sets with respect to $k$-spaces in $\PG(r-1,q^n)$: if the maximum field of linearity is $\F_{p^e}$, then there are (many) $(p^e+1)$-secants to the set. This observation shows that assuming that there is $(q+1)$-secant in the case of an $\F_q$-linear blocking set in $\PG(2,q^n)$, is equivalent to assuming that the maximum field of linearity is $\F_q$. So we see that if the linearity conjecture for blocking sets holds, then Theorem \ref{hoofd2} proves Conjecture \ref{conjecturemin}.
\end{remark}

\begin{remark} We know that every linear set $L_U$ can be written as a linear set $L_{U'}$ that contains at least one point of weight $1$. However, in Theorem \ref{hoofd2}, we cannot replace the condition ``there is a $(q+1)$-secant'' with the condition ``containing a point of weight $1$'' which we used in Theorem \ref{hoofd}. For example, a subplane $\PG(2,q^2)$ of $\PG(2,q^4)$ can be written as $\B(\mu)$ where $\mu$ is a $4$-space in $\PG(11,q)$ that meets a certain $7$-dimensional space spanned by elements of $\S$ in a $3$-space $\pi$ which intersects every element of $\S$ in a line. We see that $\B(\mu)$ has $q^4+q^2+1<q^4+q^3+1$ elements, and does contain $q^4$ points of weight one. Note that in this case, the maximum field of linearity is $\F_{q^2}$.
%%%opletten, kunnne we dit toch niet bewijzen door a) te tellen b) uit te breiden naar blocking set?
\end{remark}

\begin{remark}\label{ambetant}
Note that, for general sets, from the condition ``every line meets a set in $1$ mod $p^e$ points, but not every line meets in $1$ mod $p^{e+1}$ points'' does not need to follow that $e$ is a divisor of $t$, and hence, that $\F_{p^t}$ has a subfield $\F_{p^e}$.
For an example of this behaviour, consider $L$ to be a subline $\PG(1,q^3)$ in $\PG(1,q^9)$. By field reduction, $L$ corresponds to a set $T$ of $q^3+1$ elements of a Desarguesian $8$-spread $\S$ in $\Pi=\PG(17,q)$ such that there is a $5$-dimensional space $\mu$ of $\PG(17,q)$ meeting each element of $T$ in a plane (the $q^3+1$ planes form a Desarguesian subspread of $\mu$). Now let $\mu'$ be a hyperplane of $\mu$, then $\B(\mu')$ consists of the $q^3+1$ elements of $T$; there is one element of $T$ that meets $\mu'$ in a plane, and all other elements of $T$ meet $\mu'$ in a line. Now embed $\Pi$ in $\PG(26,q)$ and extend the Desarguesian spread $\S$ in $\Pi$ to a Desarguesian $8$-spread in $\PG(26,q)$. Take $\pi$ to be a $5$-space that meets $\Pi$ in $\mu'$. Then $\B(\pi)$ is an $\F_q$-linear set in $\PG(2,q^9)$ of rank $6$ which has size $q^5+q^3+1<q^5+q^4+1$. Moreover, a line through two points of $\B(\pi)$ meets $\B(\pi)$ in $1$ mod $q^2$ points, and there are $(1+q^2)$-secants to this set. However, $2$ is not a divisor of $9$. Note that in this case, $\B(\mu')$ has only points of weight at least $2$. 
\end{remark}

\begin{remark} \label{allesmeer} Remark \ref{ambetant} leads us to a crucial point for a possible extension of Theorem \ref{hoofd2} to general dimension, without having to impose heavy conditions on the point set as in Theorem \ref{hoofd3}. Namely, it would be useful to deduce whether or not the following holds:  if an $\F_q$-linear set of rank $k$ $L_U$ has only points of weight at least $2$, is it then true that $L_U$ is an $\F_{q^i}$-linear set for some $i>1$? It follows from \cite{BBB} that this statement is true for $\F_q$-linear sets of rank $n$ in $\PG(1,q^n)$.
\end{remark}

If we impose the assumption that there is a hyperplane of $\PG(r-1,q^n)$ that meets the linear set in $\frac{q^{r-1}-1}{q-1}$ points that span this hyperplane, then it is clear that we can repeat the argument of Theorem \ref{hoofd2} and, by induction, obtain the following Theorem:

\begin{theorem} \label{hoofd3} Let $L$ be an $\F_q$-linear set of rank $k$ spanning $\PG(r-1,q^n)$ (hence $k\geq r$) such that there is at least one hyperplane $\Pi$ of $\PG(r-1,q^n)$ meeting $L$ in exactly $\frac{q^{r-1}-1}{q-1}$ points which span $\Pi$, then $L$ contains at least $q^{k-1}+q^{k-2}+\ldots+q^{k-r+1}+1$ points.
\end{theorem}

Of course, ideally, we would like to obtain a lower bound for $\F_q$-linear sets of rank $k$ which span $\PG(r-1,q^n)$, where this condition is removed and replace by a different condition.

\begin{remark} In the case $r=3$, the imposed condition for Theorem \ref{hoofd2} is that there is one line meeting the linear set in an $\F_q$-subline. It is not clear to the authors whether it is possible to have an $\F_q$-linear set in $\PG(2,q^n)$ such that every line meets it in $1$ mod $q$ or $0$ points, but not in $1$ mod $q^2$ or $0$ points, which does not admit a $(q+1)$-secant.  As said before, it follows from the work of \cite[Corollary 5.2]{peter1} for blocking sets that it is only possible for this situation to occur for $\F_q$-linear sets of rank $k<n$.
\end{remark}

\section*{Acknowledgement}

Jan De Beule acknowledges the Research Foundation Flanders (Belgium) (FWO) for a travel grant (grant number K202718N).

\noindent Jan De Beule\\
Vrije Universiteit Brussel\\
Department of Mathematics\\
Pleinlaan 2\\
B--1050 Brussel\\
Belgium\\
jan@debeule.eu\\

\noindent Geertrui Van de Voorde\\
University of Canterbury\\
School of Mathematics and Statistics\\
Private Bag 4800\\
8140 Christchurch\\
New Zealand\\
geertrui.vandevoorde@canterbury.ac.nz


\begin{thebibliography}{1}

\bibitem{Ball2003}
S.~Ball.
\newblock The number of directions determined by a function over a finite
  field.
\newblock {\em J. Combin. Theory Ser. A}, 104(2):341--350, 2003.

\bibitem{BBB}
A.~Blokhuis, S.~Ball, A.~E. Brouwer, L.~Storme, and T.~Sz\H{o}nyi.
\newblock On the number of slopes of the graph of a function defined on a
  finite field.
\newblock {\em J. Combin. Theory Ser. A}, 86(1):187--196, 1999.

\bibitem{olga2}
G.~Bonoli and O.~Polverino.
\newblock {$\Bbb F_q$}-linear blocking sets in {${\rm PG}(2,q^4)$}.
\newblock {\em Innov. Incidence Geom.}, 2:35--56, 2005.

\bibitem{peter}
Sz.~L. Fancsali, P.~Sziklai, and M.~Tak\'ats.
\newblock The number of directions determined by less than {$q$} points.
\newblock {\em J. Algebraic Combin.}, 37(1):27--37, 2013.

\bibitem{FQ11}
M.~Lavrauw and G.~Van~de Voorde.
\newblock On linear sets on a projective line.
\newblock {\em Des. Codes Cryptogr.}, 56(2-3):89--104, 2010.

\bibitem{lidl}
R.~Lidl and H.~Niederreiter.
\newblock {\em Finite fields}, volume~20 of {\em Encyclopedia of Mathematics
  and its Applications}.
\newblock Addison-Wesley Publishing Company, Advanced Book Program, Reading,
  MA, 1983.
\newblock With a foreword by P. M. Cohn.

\bibitem{ore}
O.~Ore.
\newblock On a special class of polynomials.
\newblock {\em Trans. Amer. Math. Soc.}, 35(3):559--584, 1933.

\bibitem{olga}
O.~Polverino.
\newblock Linear sets in finite projective spaces.
\newblock {\em Discrete Math.}, 310(22):3096--3107, 2010.

\bibitem{peter1}
P.~Sziklai.
\newblock On small blocking sets and their linearity.
\newblock {\em J. Combin. Theory Ser. A}, 115(7):1167--1182, 2008.

\end{thebibliography}
\end{document}